\documentclass{amsart}
\usepackage{amsmath,amsthm,amsfonts,amssymb,latexsym,mathrsfs,graphicx}
\usepackage{pgf,tikz}
\usepackage{ulem}
\usetikzlibrary{arrows}

\usepackage{hyperref}

\usepackage{enumerate}
\usepackage[shortlabels]{enumitem}
\usepackage{color}

\usepackage[latin1]{inputenc}
\usepackage{bm} 

\usepackage{pstricks-add}
\usepackage{graphicx}
\usepackage{pgf,tikz,pgfplots}
\usepackage{mathrsfs}
\usetikzlibrary{arrows}
\definecolor{ududff}{rgb}{0.30196078431372547,0.30196078431372547,1}
\usepackage{amssymb}

\newtheorem{theorem}{Theorem}[section]

\newtheorem{lemma}[theorem]{Lemma}
\newtheorem{proposition}[theorem]{Proposition}
\newtheorem*{ThmA}{Theorem A}
\newtheorem*{ThmB}{Theorem B}

\newtheorem*{ThmC}{Theorem C}
\newtheorem*{ThmD}{Theorem D}

\theoremstyle{definition}

\newtheorem{remark}[theorem]{Remark}

\newenvironment{enumeratei}{\begin{enumerate}[\upshape (a)]}
	{\end{enumerate}}

\def\cent#1#2{{\bf C}_{#1}(#2)}

\def\oh#1#2{{\bf O}_{#1}(#2)}
\def\zent#1{{\bf Z}(#1)}

\def\fit#1{{\bf F}(#1)}
\def\frat#1{{\bf \Phi}(#1)}

\newcommand{\F}{{\mathbb F}}

\def\cent#1#2{{\bf C}_{#1}(#2)}

\def\norm#1#2{{\bf N}_{#1}(#2)}
\def\oh#1#2{{\bf O}_{#1}(#2)}

\def\zent#1{{\bf Z}(#1)}

\def\fit#1{{\bf F}(#1)}

\def\cent#1#2{{\bf C}_{#1}(#2)}

\def\oh#1#2{{\bf O}_{#1}(#2)}

\def\zent#1{{\bf Z}(#1)}

\def\norm#1#2{{\bf N}_{#1}(#2)}

\mathchardef\coso="2023

\begin{document}
	
\title[Equal order pairs]{Group cosets with all elements of equal order}
	

\author[A. R. Camina]{A. R. Camina}
\address{Alan R. Camina,  School of Engineering, Mathematics and Physics, University of East Anglia, Norwich,  NR4 7TJ, United Kingdom.}
\email{alan.camina@gmail.com}
	
\author[R. D. Camina]{R. D. Camina}
\address{Rachel D. Camina, Fitzwilliam College, Cambridge, CB3 0DG, United Kingdom.}
\email{rdc26@cam.ac.uk}	
	
\author[M. L. Lewis]{Mark L. Lewis}
\address{Mark L. Lewis, Department of Mathematical Sciences, Kent State University, Kent, OH 44266, USA.}
\email{lewis@math.kent.edu}

\author [E. Pacifici]{E. Pacifici}
\address{Emanuele Pacifici, Dipartimento di Matematica e Informatica U. Dini, Universit\`a degli Studi di Firenze, Viale Morgagni 67/A, 50134 Firenze, Italy.}
\email{emanuele.pacifici@unifi.it}

\author [L. Sanus]{L. Sanus}
\address{Lucia Sanus, Departament de Matem\`atiques, 
Universitat de Val\`encia,
46100 Burjassot, Val\`encia, Spain. }
\email{lucia.sanus@uv.es}

\author [M. Vergani]{M. Vergani}
\address{Marco Vergani, Dipartimento di Matematica e Informatica U. Dini, Universit\`a degli Studi di Firenze, Viale Morgagni 67/A, 50134 Firenze, Italy.}
\email{marco.vergani@unifi.it}


\thanks{The fourth and the sixth authors are partially supported by INdAM-GNSAGA and by the European Union-Next Generation EU, Missione 4 Componente 1, CUP B53D23009410006, PRIN 2022 2022PSTWLB - Group Theory and Applications. The research of the fifth  author is supported by Ministerio de Ciencia e Innovaci\'on (PID2022-137612NB-I00 funded by MCIN/AEI/10.13039/501100011033 and `ERDF A way of making Europe')}
\thanks{This article was initiated while some of the authors were visiting the Dipartimento di Matematica e Informatica at the Università degli Studi di Firenze. We would like to thank the department for its hospitality.}	


\begin{abstract}
Let $G$ be a finite group and $N$ a proper, nontrivial, normal subgroup of $G$. If, for every element $x$ of $G$ not lying in $N$, the elements in the coset $xN$ all have the same order as $x$, then we say that $(G,N)$ is an {\it{equal order pair}}. This generalizes the concept of a Camina pair, that was introduced by the first author. In the present paper we study several properties of equal order pairs, showing that in many respects they resemble Camina pairs, but with some important differences.
\end{abstract}
	
\maketitle
	
\section{Introduction}\label{Sect 1}

Every group considered in this paper is assumed to be a finite group. Let $G$ be a group and let $N$ be a proper, nontrivial, normal subgroup of $G$: the pair $(G,N)$ is called a {\it{Camina pair}} if, for every element $x$ of $G$ not lying in $N$, the coset $xN$ is contained in the conjugacy class of $x$ in $G$. Camina pairs are introduced in a paper by the first author (\cite{camina}) and then extensively studied by several researchers; it is known that, if $(G,N)$ is a Camina pair, then either $G$ is a Frobenius group with kernel $N$ (which happens if and only if the orders of $N$ and $G/N$ are coprime) or at least one among $N$ and $G/N$ is a group of prime-power order (see \cite[Theorem~2 and Proposition~1]{camina}). Moreover, D. Chillag, A. Mann and C. Scoppola prove in \cite[Lemma~2.1]{CMS} that if $(G,N)$ is a Camina pair where $G/N$ is solvable, then $G$ itself is solvable. In particular, if $(G,N)$ is any Camina pair then $N$ is solvable. We refer the reader to the survey \cite{marksurvey} for an exposition of many known results concerning Camina pairs, where also some generalizations are discussed. One interesting extension of the concept of a Camina pair is treated in a recent paper by Z. Akhlaghi and A. Beltr\'an (\cite{beltran}); in that paper the authors study the case when every nontrivial coset $xN$ is contained in the union of the conjugacy class of $x$ and that of $x^{-1}$.
	
Note that  if $(G,N)$ is a Camina pair, or if it satisfies the condition of \cite{beltran}, then the elements in a nontrivial coset $xN$ all have the same order. Our aim in this paper is to explore pairs that have this property, so we introduce the following concept. 

\bigskip
\noindent
{\bf Definition:} 
Let $G$ be a finite group and let $N$ be a proper, nontrivial, normal subgroup of $G$. We say that the pair $(G,N)$ is an {\it equal order pair} if, for every $x \in G \setminus N$, the elements in the coset $xN$ all have the same order as $x$.
\medskip

Straightforward examples are the pairs $(G,N)$ where $G$ is an elementary abelian $p$-group or a cyclic $p$-group for a prime number $p$, and $N$ is any proper nontrivial subgroup of $G$. In these cases, $(G,N)$ is clearly not a Camina pair.   In fact, these are not even pairs that satisfy the property studied in \cite{beltran} except when $G$ is cyclic of order $4$.  In Section \ref{Sect 2}, we will provide other examples of equal order pairs.  We also observe that, unlike Camina pairs, equal order pairs are not well behaved with respect to taking factor groups: if $(G,N)$ is an equal order pair and $M$ is a normal subgroup of $G$ properly contained in $N$, then it is not necessarily true that $(G/M,N/M)$ is an equal order pair. Nevertheless, as we will see, several techniques and arguments introduced in \cite{camina} for Camina pairs can be adapted to the present context.

After discussing some general properties, we will prove in Proposition~\ref{nilpotent} that if a {\it{nilpotent}} group $G$ has a subgroup $N$ for which $(G,N)$ is an equal order pair, then $G$ has prime-power order (and in the case when $G$ is abelian, we will describe the relevant subgroups $N$ which can occur in an equal order pair). The following result includes this observation.

\begin{ThmA}
Let $(G,N)$ be an equal order pair. Then, denoting by $\fit G$ the Fitting subgroup of $G$, either $N$ is contained in $\fit G$ or $\fit G$ is contained in $N$. Moreover, if $N$ is properly contained in $\fit G$, then $\fit G$ has prime-power order.
\end{ThmA}

In light of Theorem~A, the analysis of equal order pairs is naturally divided into two parts, depending on whether the subgroup $N$ is nilpotent (i.e., it is contained in $\fit G$) or not. In the former case, this analysis turns out to be closely related to the notion of {\it{$p'$-semiregularity}} of a group, as introduced by P.~Fleischmann, W. Lempken and P. H. Tiep in \cite{FLT}: for a given prime number $p$ and a field $\F$, a finite group $G$ is called $p'$-semiregular over $\F$ if there exists a finite-dimensional $\F G$-module $V$ such that every nontrivial $p'$-element of $G$ does not centralize any nontrivial element of $V$. As we will see in Proposition~\ref{p'semiregular}, if $(G,N)$ is an equal order pair such that $N$ is nilpotent, then $G/N$ is a $p'$-semiregular group over the field ${\rm {GF}}(p)$ of $p$~elements for every prime divisor $p$ of $|N|$. The conclusions about the Sylow subgroups of $G/N$ provided by the following result are a consequence of this fact, although they can be easily proved directly. 

\begin{ThmB}
Let $(G,N)$ be an equal order pair such that $N$ is nilpotent.
\begin{enumeratei}
\item Assume that $N$ is a $p$-group for a prime $p$. If $H$ is a nontrivial $p'$-subgroup of $G$, then $NH$ is a Frobenius group with kernel $N$ and $H$ as a Frobenius complement. In particular, for every prime $q \neq p$, the Sylow $q$-subgroups of $G$ are cyclic or generalized quaternion groups.
\item Assume that $N$ is not a group of prime-power order. Then we have $N = \fit G$. Moreover, if $q$ is a prime divisor of $|G/N|$ and $Q$ is a Sylow $q$-subgroup of $G$, then $NQ/(N\cap Q)$ is a Frobenius group with kernel $N/(N\cap Q)$ and $Q/(N\cap Q)$ as a Frobenius complement. In particular, the Sylow subgroups of $G/N$ are all cyclic or generalized quaternion groups.
\end{enumeratei}
\end{ThmB}

In the situation of Theorem~B, the structure of $G/N$ is therefore well understood: finite groups whose Sylow subgroups are all cyclic or generalized quaternion have been classified by H. Zassenhaus in \cite{Z}, whereas Section~2 of \cite{FLP} provides many structural features of groups satisfying the same property except for one prime.  For instance, we will see that if $(G,N)$ is an equal order pair where $G$ is solvable and $N$ is nilpotent, then the Fitting height of $G$ is at most~$4$ (Proposition~\ref{fittingheight}); on the other hand, we will also see in Remark~\ref{gag} that the derived length of $G/N$ cannot be bounded.

We point out that if $S$ is a group acting $p'$-semiregularly on the $\F S$-module $V$ and $G$ is the semidirect product $V\rtimes S$ formed accordingly, then $(G,V)$ is not necessarily an equal order pair. In fact, let $q=p^n$ be a power of the prime number $p$, where $n$ is a positive integer: then the group $S={\rm{SL}}_2(q)$ acts $p'$-semiregularly on its natural $2$-dimensional module $V$ over ${\rm {GF}}(q)$. But, setting $G=V\rtimes S$, it turns out that $(G,V)$ is an equal order pair if and only if $p$ is odd (see Remark~\ref{natural}). 
On the one hand, this produces many examples of $p'$-semiregular actions not yielding equal order pairs; yet on the other hand, it also produces an infinite family of equal order pairs $(G,N)$ {\it{where $G/N$ is nonsolvable}} and $G$ is not a Frobenius group with kernel $N$. Note that producing examples of non-Frobenius Camina pairs $(G,N)$ where $G/N$ is nonsolvable appears to be considerably harder; the first example of this kind has been described in \cite{nonsolv}, and others can be found in Theorem~1.20 of \cite{GGLMNT}. 

Another feature of non-Frobenius Camina pairs, which we already recalled, is the fact that either $N$ or $G/N$ must have prime-power order; regarding non-Frobenius equal order pairs, it is still true that when $|G/N|$ is not a prime power, then $N$ must be nilpotent (see Theorem~D), but $N$ can have an order divisible by more than one prime. For example, it is possible to define an action of ${\rm SL}_2(5)$ on an abelian group $V$ of squarefree exponent and order $3^4\cdot 5^2$ so that $(V\rtimes{\rm SL}_2(5),V)$ is an equal order pair (see Remark~\ref{mixed}). The following result, Theorem~C, whose proof uses the $p'$-semiregularity condition and relies on Theorem~4.1 of \cite{FLT}, gives some necessary conditions for this to occur. 

To introduce the statement of Theorem~C, we define a particular set of prime numbers $\mathcal{R}$ (see \cite[Section~4]{FLT}). This is the set of all primes $r$ satisfying both of the following conditions:

\begin{enumeratei}
\item $r=2^a \cdot 3^b + 1$, for integers $a \geq 2$ and $b \geq 0$.
\item $(r+1)/2$ is a prime.
\end{enumeratei}
Note that $5 \in \mathcal{R}$.
Recall also, for a group $G$, that the {\it{solvable residual}} $K$ of $G$ is defined to be the smallest normal subgroup of $G$ whose quotient is solvable and the {\it solvable radical} $R$ of $G$ is defined to be the largest normal subgroup of $G$ that is solvable.  When $n$ is an integer, we define $\pi (n)$ to be the set of primes dividing $n$, and when $N$ is a subgroup of $G$, we write $\pi (N)$ for $\pi (|N|)$.

\begin{ThmC}
Let $(G,N)$ be an equal order pair such that $G/N$ is nonsolvable and $N$ does not have prime-power order. Denote by $R/N$ and $K/N$ the solvable radical and the solvable residual of $G/N$, respectively. Then there exists a prime $p$ in $\mathcal{R} \cup \{ 7,17 \}$ such that the following conclusions hold.
\begin{enumeratei}
\item $G/R$ is isomorphic either to ${\rm PSL}_2 (p)$ or to ${\rm PGL}_2 (p)$.
\item $K/N$ is isomorphic to ${\rm SL}_2 (p)$.
\item If there exists a prime $r \in \pi (N)$ such that $K/N$ has a normal Sylow $r$-subgroup, then $p$ is necessarily $5$. Otherwise, we have $\pi(N) = \{ 3,p \}$ when $p \neq 5$, whereas $\pi(N) \subseteq \{ 2,3,5 \}$ when $p = 5$.
\end{enumeratei}
\end{ThmC}

Next, we consider the other possibility given by Theorem~A for an equal order pair $(G,N)$; namely the situation when $\fit G$ is properly contained in $N$. In this case, we will see that every element of $G$ not lying in $N$ has $p$-power order for a prime number $p$.  This situation has been studied by the third author in \cite{lewis}, and in \cite{SJ}.  Before stating the next result, we recall that if a group $G$ has a proper, nontrivial subgroup $H$ and a proper normal subgroup $L$ of $H$ such that $H\cap H^g\subseteq L$ for all elements $g\in G\setminus H$, then $(G,H,L)$ is called a {\it{Frobenius-Wielandt triple}} (see \cite[Section 3]{lewis}).

\begin{ThmD}
Let $(G,N)$ be an equal order pair such that $N$ is not nilpotent. Then the following conclusions hold.
\begin{enumeratei}
\item There exists a prime number $p$ such that every element in $G \setminus N$ is a $p$-element; in particular, $G/N$ is a $p$-group. 
\item If $P$ is a Sylow $p$-subgroup of $G$, then $(G,P,P\cap N)$ is a Frobenius-Wielandt triple.
\item If $N$ is nonsolvable, then $p=2$ and every nonsolvable chief factor $U/V$ of $G$ such that $U\subseteq N$ is isomorphic to ${\rm{PSL}}_2(3^{2^r})$ for some $r>0$. 
\end{enumeratei}
\end{ThmD}

As another difference with Camina pairs, equal order pairs $(G,N)$ as in (c) of Theorem D do exist.  In fact, there exists a split extension $G$ of the alternating group ${\rm{A}}_6\cong{\rm{PSL}}_2(3^{2})$ by a cyclic group of order $8$ such that $\zent G$ is cyclic of order $4$, $G/\zent G$ is isomorphic to the Mathieu group ${\rm M}_{10}$ and, denoting by $N$ the unique subgroup of index $2$ in $G$, the pair $(G,N)$ is an equal order pair.  As we will see in Proposition~\ref{minimalnonsolvable}, a group $G$ which is minimal with respect to the property of having a nonsolvable subgroup $N$ such that $(G,N)$ is an equal order pair is necessarily similar to the above example; namely, denoting by $R$ the solvable radical of such a group $G$, we have $G/R\cong{\rm M}_{10}$ and $R$ is not a minimal normal subgroup of $G$. In general, since ${\rm{PSL}}_2(3^f)$ has a subgroup isomorphic to ${\rm A}_6$ whenever $f$ is even, an equal order pair as in (c) of Theorem~D always has a section (contained in $N$) isomorphic to ${\rm A}_6$.


\section{Preliminary results and some examples}\label{Sect 2}

We start this section by proving some preliminary results; as we will see, many of the techniques from \cite{beltran} and~\cite{camina} remain valid in the context of equal order pairs. The following lemma is an easy but useful observation.

\begin{lemma} \label{subgroups} 
Let $(G,N)$ be an equal order pair. If $X$ is a subgroup of $G$ not contained in $N$ and $X \cap N$ is nontrivial, then $(X, X \cap N)$ is also an equal order pair.
\end{lemma}

\begin{proof}
Clearly, $X \cap N$ is a proper normal subgroup of $X$. If $x \in X \setminus(X \cap N)$ then $x$ lies in $G \setminus N$; since, for every $y \in X$ such that $y( X \cap N) = x(X \cap N)$, we have $yN = xN$, our assumption on $(G,N)$ yields that the order of $y$ is the same as the order of $x$.
\end{proof}

Next, for a given equal order pair $(G,N)$, we describe the orders of elements outside $N$; in fact, inspired by \cite{guralnick}, we consider the weaker condition that all elements of $xN$ have the same order {\it for a single} element $x \in G\setminus N$. The proofs of Lemma \ref{two} and (a), (b), (d) of Proposition \ref{pq} closely follow arguments in \cite{camina}.  When $G$ is a group and we have an element $g \in G$, we write ${\rm o}(g)$ for the order of $g$.

\begin{lemma}\label{two} 
Let $G$ be a group, let $N$ a normal subgroup of $G$ and fix an element $x \in G \setminus N$. Suppose ${\rm o}(xn) = {\rm o}(x)$ for all $n \in N$.  If $x$ has order $m$ modulo $N$, then the order of $x$ is only divisible by primes that divide $m$. In particular, if $y \in \cent N x$ has prime order $p$, then $p$ divides $m$. 
\end{lemma}

\begin{proof} 
Let $\pi = \pi (m)$ denote the set of prime divisors of $m$, $\pi'$ the set of primes not lying in $\pi$, and assume that the order of $x$ is not a $\pi$-number. Write $x = uv$ where $u$ is a $\pi$-element, $v$ is a $\pi'$-element, and $[u,v]=1$. As the order of $xN$ is a $\pi$-number, we clearly have $v\in N$. But then $u$ and $uv$ have the same order, a contradiction. Thus the order of $x$ is only divisible by primes that divide $m$. Finally suppose $y \in\cent Nx$ has order $p$.  Since ${\rm o} (xy) ={\rm o} (x)$, it follows that $1 = (xy)^m= x^my^m = y^m$ and the result follows.
\end{proof}

\begin{remark}\label{remFrobenius}
If $(G,N)$ is an equal order pair such that $N$ and $G/N$ have coprime orders then, by the Schur-Zassenhaus Theorem, $N$ has a complement $H$ in $G$.  Moreover, by the previous lemma, for every element $h \in H$ we have $\cent N h = 1$. Thus, $G$ is a Frobenius group with kernel $N$.  Note that since every normal subgroup of an elementary abelian group yields both an equal order pair and a complement, we cannot generalize the theorem of Chillag and MacDonald \cite[Proposition 3.2]{ChMa} that when $(G,N)$ is a Camina pair that is complemented, then $G$ is a Frobenius group with the Frobenius kernel $N$.
\end{remark}

The following result provides conditions on the orders of elements outside of $N$ that ensure that $N$ is abelian, nilpotent or solvable.

\begin{proposition}\label{pq} 
Let $G$ be a group, $N$ a normal subgroup of $G$ and $x \in G \setminus N$. Suppose ${\rm o} (xn) = {\rm o}(x)$ for all $n \in N$. Then the following conclusions hold. 

\begin{enumeratei}
\item If ${\rm o}(x) = 2$ then $N$ is abelian.

\item If ${\rm o}(x) = p$ for some prime $p$, then $N$ is nilpotent. 

\item If ${\rm o}(xN)$ is odd, then $N$ is solvable.

\item Let $(\langle x \rangle N,N)$ be an equal order pair. If $m = {\rm o}(xN)$ is not a prime power, then ${\rm o}(x) = m$ and $N$ is nilpotent.
\end{enumeratei}
\end{proposition}

\begin{proof} 
In the setting of (a), for every $n$ in $N$ we have $xnxn=1$, hence $n^x = x^{-1}nx = xnx = n^{-1}$. So the inversion map is an automorphism of $N$, and it follows that $N$ is abelian.

Next, we prove (b). For $n \in N$ we have 
$$nn^{x^{-1}}n^{x^{-2}} \ldots \, n^{x^{-(p-1)}} = (nx)^p =1$$
(as $nx\in Nx= xN$). It follows from \cite[V, 8.13]{huppert} that $N$ is nilpotent.

Assuming the hypothesis in (c), by Lemma \ref{two} the order of $x$ is odd. Therefore, conclusion (c) follows directly from \cite[Theorem B(c)]{guralnick}.

Finally, assume the setting of (d). Then there exist coprime integers $h$, $k$ larger than $1$ such that $m=hk$. By Lemma~\ref{two} we have that $\cent N{x^h}$ is a $\pi(k)$-group and $\cent N{x^k}$ is a $\pi(h)$-group. In particular, we get $\cent N x \leq \cent N{x^h} \cap \cent N{x^k} = 1$ and, since $x^{m} \in \cent N x$, we deduce that ${\rm o}(x)=m$. Now, taking a prime divisor $p$ of $m$, we see that $x^{m/p}$ is an element of order $p$ lying outside $N$, and the result follows from (b).
\end{proof}

By conclusion (d) of the above proposition, if $(G,N)$ is an equal order pair and $N$ is not nilpotent, then every element of $G/N$ has prime-power order; in fact, combining this with Lemma~\ref{two}, we get the stronger conclusion that every element of $G\setminus N$ has prime power order. Assuming $G/N$ solvable, G.~Higman proved that in this situation $|G:N|$ is divisible by at most two distinct primes; moreover, if $|G:N|$ is divisible by exactly two distinct primes, then $G/N$ is either a Frobenius group or a $2$-Frobenius group (see \cite[Theorem~1]{Higman57}).  We will define $2$-Frobenius groups in Section \ref{Sect 6}. 
However, we will see in Theorem~D that if $N$ is not nilpotent, then $G/N$ is necessarily solvable, in fact $G/N$ is a $p$-group for a prime $p$.

\medskip
In the remaining part of this section, we provide some examples of equal order pairs, together with some easy tools to produce them. In the following result, we see that equal order pairs work well with direct products under certain conditions; this will allow us to construct new equal order pairs as products of existing ones.

\begin{lemma}\label{dir prod pairs}
Suppose that $(G_1,N_1)$ and $(G_2,N_2)$ are equal order pairs where all elements in $G_1 \setminus N_1$ have order $m_1$, all elements in $G_2 \setminus N_2$ have order $m_2$, the exponent of $N_2$ divides $m_1$ and the exponent of $N_1$ divides $m_2$.  Then $(G_1 \times G_2,N_1 \times N_2)$ is an equal order pair.
\end{lemma}

\begin{proof}
Given $(g_1,g_2) \in (G_1 \times G_2) \setminus (N_1 \times N_2)$, we show that all elements in the coset $(g_1,g_2)(N_1 \times N_2)$ have the same order.   

Observing that ${\rm o}((g_1,g_2)) = {\rm lcm} ({\rm o}(g_1),{\rm o}(g_2))$ and choosing $n_1 \in N_1$, $n_2 \in N_2$, assume first $g_1 \in G_1 \setminus N_1$ and $g_2 \in N_2$.  Then $g_1$ and $g_1n_1$ have order $m_1$, whereas $g_2$ and $g_2n_2$ have orders dividing $m_1$.  Thus, ${\rm o}((g_1,g_2)) = {\rm o} ((g_1 n_1, g_2n_2)) = m_1$.
The case when $g_1 \in N_1$ and $g_2 \in G_2 \setminus N_2$ is similar (and we get ${\rm o}((g_1,g_2)) = m_2$), so we assume $g_1 \in G_1 \setminus N_1$ and $g_2 \in G_2 \setminus N_2$. In this case, $g_1$ and $g_1n_1$ have order $m_1$, whereas $g_2$ and $g_2n_2$ have order $m_2$.  Thus, ${\rm o}((g_1,g_2)) = {\rm o}((g_1 n_1, g_2n_2)) = {\rm lcm} (m_1, m_2)$.
\end{proof}

In view of the previous lemma, we see that $({\rm S}_3 \times {\rm A}_4, \;{\rm A}_3 \times \oh 2{{\rm A}_4})$ is an equal order pair. More generally, for $p,q$ prime numbers, let $G_1$ be a Frobenius group whose kernel $N_1$ is an elementary abelian $q$-group and whose Frobenius complements have order $p$; also, let $G_2$ be a Frobenius group whose kernel $N_2$ is an elementary abelian $p$-group and whose Frobenius complements have order $q$. Then $(G_1 \times G_2, N_1\times N_2)$ is an equal order pair. For instance, it is easy to construct a group of this kind where $|G_1| = 3\cdot 5^2$ and $|G_2| = 3^4 \cdot 5$.  Note that, for  these equal order pairs, $N = \fit G$ and $G$ is not a Frobenius group.  

\smallskip
The next result can be used to produce examples of equal order pairs starting from a Frobenius group. 

\begin{lemma} \label{generalize j}
Let $G$ be a Frobenius group with kernel $K$, and let $H$ be a Frobenius complement of $G$; also, let $N$ be a normal subgroup of $G$ properly containing $K$ and define $M = H \cap N$. Then $(H,M)$ is an equal order pair if and only if $(G,N)$ is an equal order pair.  	
\end{lemma}

\begin{proof}
Note that we have $N=KM$, so $M$ is nontrivial. If $(G,N)$ is an equal order pair then, by definition, $N$ is a proper subgroup of $G$; therefore $H$ is not contained in $N$ and Lemma~\ref{subgroups} yields that $(H,M)$ is an equal order pair.
	
Assume now that $(H,M)$ is an equal order pair.  Since $G = HN$, we have to show that for every $h \in H \setminus N$ all elements in $hN$ have the same order.  We know that, for $y \in M$, the conjugacy class of $hy$ in $G$ is $hyK$.  Hence $hN$ is partitioned by the conjugacy classes $hyK$ as $y$ runs over $M$, and therefore every element in $hN$ is conjugate to an element of $hM$. Since $(H,M)$ is an equal order pair, every element in $hM$ has the same order and our claim is proved.  
\end{proof}

The above lemma applies whenever a Frobenius group has a cyclic Frobenius complement of prime-power order. For instance, if $G$ is the Frobenius group of order $20$ and $N$ is its subgroup of index $2$ (isomorphic to the dihedral group ${\rm D}_{10}$), then $(G,N)$ is an equal order pair since a Frobenius complement $C$ for $G$ is cyclic of order $4$, and $(C,C \cap N)$ is an equal order pair; so Lemma \ref{generalize j} implies that $(G,N)$ is an equal order pair.  Notice that this yields an equal order pair $(G,N)$ where $G$ is solvable and $N$ is not nilpotent.

\smallskip

Other examples of equal order pairs arise taking into account the following easy observation: if $(G,N)$ and $(N,M)$ are equal order pairs so that $M$ is normal in $G$, then $(G,M)$ is an equal order pair. So, for instance, if $G$ is a Frobenius group whose kernel $N$ is an elementary abelian $p$-group and there exists a nontrivial normal subgroup $M$ of $G$ with $M \subseteq N$, then $(G,M)$ is an equal order pair. 

\smallskip
As the last example for this section, let $p$ be a prime number and consider a group $K$ that admits a fixed-point free automorphism $\alpha$ of order $p$; also, let $P = \langle x\rangle$ be a cyclic $p$-group. We can define an action of $P$ on $K$ by $k^x = \alpha (k)$ for every element $k \in K$, and then we set $G$ to be the semidirect product $K \rtimes P$ formed accordingly. If $M$ is the subgroup of index $p$ in $P$, then every element in $G \setminus KM$ is conjugate to an element of $P \setminus M$, and so $(G,KM)$ is an equal order pair.

\section{When $G$ is nilpotent}\label{Sect 3}

Our aim in this section is to describe equal order pairs $(G,N)$ such that the group $G$ is nilpotent.  As mentioned in the Introduction, this happens only when $G$ is in fact a group of prime-power order.

\begin{proposition}\label{nilpotent} 
Let $G$ be a nilpotent group having a subgroup $N$ such that $(G,N)$ is an equal order pair. Then $G$ is a $p$-group for some prime $p$.
\end{proposition}

\begin{proof}
As $G$ is nilpotent, we know that $N \cap \zent G$ is nontrivial. Take a prime $p$ dividing its order; if $x$ is an element outside $N$, by Lemma \ref{two} we have that $p$ divides the order of $xN$. In particular every nontrivial element of $G/N$ has order divisible by $p$, so $G/N$ is a $p$-group. Since we can do this with every prime divisor of $|N \cap \zent G|$, we get that $N\cap \zent G$ is a $p$-group as well and, since every Sylow subgroup of $N$ is normal in $G$ and intersects $\zent G$ nontrivially, we must have that $N$ is a $p$-group. It follows that $G$ is a $p$-group.
\end{proof}

The next proposition characterizes the central subgroups $N$ of a nilpotent group $G$ such that $(G,N)$ is an equal order pair. Recall that, if $G$ is a $p$-group for a prime $p$ and $i$ is a nonnegative integer, then the subgroup ${\bf{\Omega}}_i(G)$ is defined as $\langle x \in G\; :\; x^{p^i} =1 \rangle$.  This defines an ascending chain of characteristic subgroups of $G$.  

\begin{proposition} \label{nilpotent pairs}
Let $G$ be a $p$-group for a prime number $p$, and let $N$ be a proper nontrivial subgroup of $G$.  Then the following conclusions hold.
\begin{enumeratei}
\item Suppose $N$ is central in $G$. Then $(G,N)$ is an equal order pair if and only if there exists a positive integer $i$ such that ${\bf{\Omega}}_{i} (\zent G) \geq N \geq {\bf{\Omega}}_{i-1}(G)$.
\item Suppose $G$ is abelian.  Then $(G,N)$ is an equal order pair if and only if there exists a positive integer $i$ such that  ${\bf{\Omega}}_{i}(G) \geq N \geq {\bf{\Omega}}_{i-1}(G)$.
\end{enumeratei}
\end{proposition}

\begin{proof}
Assume that $(G,N)$ is an equal order pair, and let $x$ be an element of minimal order, say $p^i$, among those in $G \setminus N$. Then we clearly have ${\bf{\Omega}}_{i-1}(G) \subseteq N$. Furthermore, as all elements in $xN$ have the same order as $x$, we see that ${\rm o}(x) \geq {\rm exp}(N)$. Thus $N \subseteq {\bf{\Omega}}_{i} (\zent G)$.  

Conversely, assume ${\bf{\Omega}}_{i}(\zent G)\supseteq N\supseteq {\bf{\Omega}}_{i-1}(G)$ where $i$ is a suitable positive integer, and let $x \in G \setminus N$:  we need to show that, for every $n\in N$, the order of $xn$ is the same as the order of $x$.  Since $x$ and $n$ commute, we know that ${\rm o}(xn) = {\rm lcm} ({\rm o}(x),{\rm o}(n))$. Also, the conditions ${\bf{\Omega}}_{i-1} (G) \subseteq N$ and $x \in G \setminus N$ imply that $p^i$ divides ${\rm o}(x)$.   On the other hand, ${\rm o}(n)$ divides $p^i$ because $N \subseteq {\rm{\Omega}}_i (\zent G)$, and so ${\rm o}(n)$ divides ${\rm o}(x)$.  It follows that ${\rm o}(xn) = {\rm lcm} ({\rm o}(x),o(n)) = {\rm o}(x)$, as wanted. Conclusion (b) is an immediate consequence of (a).
\end{proof}

In view of the previous result, if $G$ is a $p$-group and $N \subseteq \zent G$ has exponent $p$, then $(G,N)$ is an equal order pair; as an example we can consider a Suzuki $2$-group $G$ and take $N = \zent G = {\bf{\Omega}}_1 (G)$ (note that Suzuki groups of Type A do not produce Camina pairs in this way; for information about Suzuki groups and in particular those of Type A, see \cite{Higman}). On the other hand, if the exponent of the whole $G$ is $p$, then clearly every proper, nontrivial, normal subgroup $N$ of $G$ yields an equal order pair $(G,N)$.

\begin{remark} \label{gag} By Theorem~6.3 of \cite{gagola}, for every $p$-group $P$ there exists a Camina pair $(X,N)$ such that $N = \zent X$ and $P$ is isomorphic to a subgroup $G/N$ of $X/N$. Since $(X,N)$ is also an equal order pair and, by Lemma~\ref{subgroups}, $(G,N)$ is an equal order pair as well, it follows that for every $p$-group $P$ there exist a $p$-group $G$ and a central subgroup $N$ of $G$ such that $(G,N)$ is an equal order pair with $P\cong G/N$. 

In other words, no restrictions can be expected on the structure of the factor group $G/N$ when $G$ is a $p$-group and $(G,N)$ is an equal order pair. In particular, we cannot bound the derived length of a solvable group $G$ which has a subgroup $N$ such that $(G,N)$ is an equal order pair.
\end{remark}

Moving to the case when the subgroup $N$ is not necessarily central, we briefly discuss a couple of situations. The {\it{Hughes subgroup}} of a $p$-group $G$ is defined as the subgroup $H$ of $G$ generated by the elements of order divisible by $p^2$; it is clearly a characteristic subgroup of $G$.  (We refer the reader to \cite{HVL} and the references within for the history of the Hughes' subgroups problem.)  If $H$ is proper and nontrivial, then $(G,H)$ is an equal order pair since every element in $G \setminus H$ has order $p$. 
 
Also, let $G$ be a $p$-group with exponent $p^n$, such that ${\bf \Omega}_{n-1} (G)$ is properly contained in $G$. Then, for any proper normal subgroup $N$ of $G$ with ${\bf \Omega}_{n-1} (G) \subseteq N$, we see that $(G,N)$ is an equal order pair.  Note that also in this case all the elements in $G \setminus N$ have the same order (namely, the exponent $p^n$ of $G$).  In fact, it is not difficult to see that we can characterize this situation as follows: if $G$ is a $p$-group of exponent $p^n$ and $N$ is a proper, nontrivial, normal subgroup of $G$, then $(G,N)$ is an equal order pair where all the elements outside $N$ have order $p^n$ if and only if ${\bf \Omega}_{n-1} (G) \subseteq N$.

 

\section{Theorem~A}\label{Sect 4}

We  work towards a proof of Theorem~A, that will be preceded with two easy lemmas.  Recall that, as mentioned in the Introduction, if a group $G$ has a proper, nontrivial subgroup $H$ and a proper normal subgroup $L$ of $H$ such that $H\cap H^g\subseteq L$ for all $g\in G \setminus H$, then $(G,H,L)$ is called a Frobenius-Wielandt triple. Wielandt formulated this definition as a generalization of Frobenius groups.  Section 3 of \cite{lewis} and the references there-in give an exposition of Frobenius-Wielandt triples.  We note that in \cite{CMS} it is shown that Frobenius-Wielandt triples are intimately connected to Camina pairs $(G,N)$ where $G$ is not a Frobenius group and $G/N$ is a $p$-group.  This first lemma shows that Frobenius-Wielandt triples arise when we have a normal subgroup such that every element outside this normal subgroup has $p$-power order.  This is one of the situations studied by the third author in \cite{lewis}, and by Shao and Jiang in \cite{SJ}.

\begin{lemma}\label{seven} 
Let $(G,N)$ be an equal order pair and $p$ a prime number. Assume that a Sylow $p$-subgroup $P$ of $G$ is not contained in $N$, and that $N$ is not a $p$-group. Then $(PN, P, P \cap N)$ is a Frobenius-Wielandt triple. 
\end{lemma}

\begin{proof} 
Observe that $(PN, N)$ is an equal order pair by Lemma \ref{subgroups}.  In particular, if $y \in PN \setminus N$,
then we have $y = xn$ for some $x \in P$ and $n \in N$ so that ${\rm o}(y) = {\rm o}(x)$; hence, $y$ is a $p$-element and it is therefore conjugate in $G$ to an element of $P$. The claim now follows by an application of Lemma~3.3 in \cite{lewis}.
\end{proof}

We now consider the possibility that, for an equal order pair $(G,N)$, there is a subgroup of $G$ that intersects $N$ trivially and centralizes $N$.  In particular, this applies if $G$ has a minimal normal subgroup outside of $N$.  

\begin{lemma}\label{product}
Let $(G,N)$ be an equal order pair. If $M$ is a nontrivial subgroup of $G$ such that $M \cap N = 1$ and $M \subseteq\cent G N$, then $M$ and $N$ are both $p$-groups for a prime $p$, and $N$ has exponent $p$.
\end{lemma}

\begin{proof}
Let $p$ be a prime divisor of $|M|$ and let $x$ be an element of $M$ having order $p$. For every element $n\in N$, we have $p = {\rm o}(x) = {\rm o}(xn) = {\rm lcm} (p,{\rm o}(n))$, hence $n^p=1$.  It follows that $N$ is a $p$-group of exponent $p$ and, since this holds for every prime divisor $p$ of $|M|$, clearly $M$ is a $p$-group as well.
\end{proof}

We are now in a position to prove Theorem~A, that we state again.

\begin{ThmA}
Let $(G,N)$ be an equal order pair. Then, denoting by $\fit G$ the Fitting subgroup of $G$, either $N$ is contained in $\fit G$ or $\fit G$ is contained in $N$. Moreover, if $N$ is properly contained in $\fit G$ then $\fit G$ has prime-power order.
\end{ThmA}

\begin{proof} 
Assume, arguing by contradiction, that neither $\fit G \subseteq N$ nor $N \subseteq \fit G$.  Assume also, for the moment, $\fit G \cap N \neq 1$.  Then, by Lemma~ \ref{subgroups}, $(\fit G, \fit G \cap N)$ is an equal order pair.  Since $\fit G$ is nilpotent, Proposition~\ref{nilpotent} yields that $\fit G = \oh p G $ is a $p$-group for a prime $p$ and hence, it is contained in any Sylow $p$-subgroup of $G$. 

Now, let $P$ be a Sylow $p$-subgroup of $G$. Since $N$ is not contained in $\fit G$, we have that $N$ is not a $p$-group; moreover, $P$ is not contained in $N$, as otherwise $\fit G$ would be contained in $N$. By Lemma~\ref{seven}, $(PN, P, P \cap N)$ is a Frobenius-Wielandt triple.  But then, for any element $x \in PN \setminus P$, we have that $P^x \cap P$ is contained in $P \cap N$. Thus, $\fit G$ is contained in $P \cap N$ and hence in $N$, contradicting our assumptions.

Finally, suppose $\fit G \cap N = 1$. Since our hypothesis implies $\fit G \neq 1$, we must have $\fit N = 1$ which implies that $N$ is not solvable.  In particular $N$ is not a $p$-group, against Lemma~\ref{product}. This contradiction completes the proof of our first claim. As for the claim that $\fit G$ is a group of prime-power order whenever it properly contains $N$, this follows immediately from Lemma~\ref{subgroups} and Proposition~\ref{nilpotent}. 
\end{proof}

\section{When $N$ is nilpotent}\label{Sect 5}

In this section, we discuss equal order pairs $(G,N)$ such that $N$ is nilpotent; as the first step, we will derive Theorem~B from the following more general result. 

\begin{proposition}\label{nilpotent N}
Let $(G,N)$ be an equal order pair, $q$ a prime divisor of $|G/N|$, $p\neq q$ a prime divisor of $|N|$ and $P$ a Sylow $p$-subgroup of $N$. Denoting by $Q_0$ a Sylow $q$-subgroup of $\norm G P$, assume that $P_0=\cent P{Q_0\cap N}$ is nontrivial. Then $P_0 Q_0/(Q_0\cap N)$ is a Frobenius group with kernel $P_0(Q_0\cap N)/(Q_0\cap N)$ and $Q_0/(Q_0\cap N)$ as a Frobenius complement. In particular, the Sylow $q$-subgroups of $G/N$ are either cyclic or generalized quaternion groups.
\end{proposition}

\begin{proof}  
By the Frattini argument we know that $G = N \norm G P$.  Note that $P_0$ is normalized by $Q_0$, and $Q_0/(Q_0\cap N)$ acts fixed-point freely on $P_0$ by Lemma \ref{two}: it follows that $P_0 Q_0/(Q_0\cap N)$ is a Frobenius group with kernel $P_0(Q_0\cap N)/(Q_0\cap N)$ and $Q_0/(Q_0 \cap N)$ as a complement. In particular, $Q_0/(Q_0\cap N)$ is cyclic or a generalized quaternion group. 

Finally, let $Q$ be a Sylow $q$-subgroup of $G$ containing $Q_0$. Since $G = N \norm G P$, it is not difficult to check that we have $QN/N \cong Q/ (Q \cap N)  \cong Q_0 /(Q_0 \cap N)$, and the proof is complete.
\end{proof}

Note that the conclusion about the Sylow $q$-subgroups of $G/N$ in the above proposition is no longer valid in general if the assumption $P_0 \neq 1$ is dropped. Consider for instance the Frobenius group $G$ of order $72$ with complements isomorphic to ${\rm Q}_8$, and let $N = G'$: then $(G,N)$ is an equal order pair (in fact, a Camina pair) by Lemma~\ref{generalize j}, we have $p=3$, $q=2$, $Q_0 \cap N \cong C_2$ acts fixed-point freely on $P$, and $G/N$ is isomorphic to $C_2 \times C_2$. 

\smallskip
We are now in a position to prove Theorem~B.

\begin{ThmB}
Let $(G,N)$ be an equal order pair such that $N$ is nilpotent.
\begin{enumeratei}
\item Assume that $N$ is a $p$-group for a prime $p$. If $H$ is a nontrivial $p'$-subgroup of $G$, then $NH$ is a Frobenius group with kernel $N$ and $H$ as a Frobenius complement. In particular, for every prime $q\neq p$, the Sylow $q$-subgroups of $G$ are cyclic or generalized quaternion groups.
\item Assume that $N$ is not a group of prime-power order. Then we have $N=\fit G$. Moreover, if $q$ is a prime divisor of $|G/N|$ and $Q$ is a Sylow $q$-subgroup of $G$, then  $NQ/(N\cap Q)$ is a Frobenius group with kernel $N/(N\cap Q)$ and $Q/(N\cap Q)$ as a Frobenius complement. In particular, the Sylow subgroups of $G/N$ are all cyclic or generalized quaternion groups.
\end{enumeratei}
\end{ThmB}

\begin{proof} Conclusion (a) is an application of the argument in Remark~\ref{remFrobenius}, taking into account that $(NH,N)$ is an equal order pair by Lemma~\ref{subgroups}. 

As regards conclusion (b), our assumption that $N$ is nilpotent but not of prime-power order yields $N = \fit G$ in view of Theorem~A.  Moreover, if $p$ is any prime divisor of $|N|$ different from $q$ (which certainly exists) and $P$ is the Sylow $p$-subgroup of $N$, then $P$ is normal in $G$, we have $\cent P{Q\cap N}=P\neq 1$ and Proposition~\ref{nilpotent N} ensures that the Sylow $q$-subgroups of $G/N$ are cyclic or generalized quaternion groups. The same proposition also yields that $PQ/(Q\cap N)$ is a Frobenius group with kernel $P(Q\cap N)/(Q\cap N)$ and $Q/(Q\cap N)$ as a Frobenius complement. Since 
$$\frac {NQ}{N \cap Q} = \left( \prod_{p \mid |N|,\;p\neq q} \frac {P(Q\cap N)}{Q\cap N} \right) \cdot \frac {Q}{Q\cap N},$$
the remaining part of our claim is easily proved.
\end{proof}

Let $G$ be a group, $p$ a prime number and $\F$ a field. According to \cite[Definition 1.1]{FLT}, $G$ is called {\it $p'$-semiregular over $\F$} if there exists a finite-dimensional $\F G$-module $V$ such that the centralizer in $V$ of every nontrivial $p'$-element of $G$ is trivial (note that if, in this situation, $p$ does not divide the order of $G$, then $V\rtimes G$ is a Frobenius group with $G$ as a Frobenius complement).  As mentioned in the Introduction, this concept is relevant in the context of equal order pairs $(G,N)$ where $N$ is nilpotent: this connection is the subject of the following result.

\begin{proposition}\label{p'semiregular} 
Let $(G,N)$ be an equal order pair such that $N$ is nilpotent, and let $\overline{G} = G/N$. Then both $\overline{G}$ and $\overline{G}/\oh{p}{\overline{G}}$ are $p'$-semiregular over ${\rm GF}(p)$ for every prime divisor $p$ of $|N|$.
\end{proposition}

\begin{proof} 
Let $p$ be a prime divisor of $|N|$ and $P$ the Sylow $p$-subgroup of $N$; observe that, denoting by $\frat P$ the Frattini subgroup of $P$, the factor group $V=P/\frat P$ can be viewed as a ${\rm GF}(p)[G/N]$-module. We claim that, for every nontrivial $p'$-element $xN$ of $G/N$, the centralizer in $V$ of $xN$ is trivial. 

For a proof by contradiction, assume there exists a nontrivial element $v\frat P$ of $V$ that is centralized by $xN$; writing $x=x_{p'}\cdot x_p$ as the product of two commuting elements of $p'$-order and $p$-power order respectively, it is clear that $x_p$ lies in $N$ and so it centralizes $v\frat P$. Thus, the $p'$-element $x_{p'}\in G\setminus N$ centralizes $v\frat P$ as well. Since, by Lemma~\ref{two}, we have $\cent{\langle x_{p'}\rangle} u=\langle x_{p'}\rangle\cap N$ for every element $u\in v\frat P$, we see that $\frat P\neq 1$ and $\langle x_{p'}\rangle$ permutes the elements of $v\frat P$ inducing orbits of equal size $s=|\langle x_{p'}\rangle/(\langle x_{p'}\rangle\cap N)|$. But $s\neq 1$ is a $p'$-number and $|v\frat P|$ is a $p$-power, so this is clearly a contradiction and our claim is proved. This shows that $\overline{G}$ is $p'$-semiregular over ${\rm GF}(p)$, and now the same conclusion holds for $\overline{G}/\oh{p}{\overline{G}}$ in view of Corollary~1.4 of \cite{FLT}.
\end{proof}

The part of Theorem~B regarding the Sylow subgroups of $G/N$ is in fact a consequence of the previous observation (see \cite[Lemma~1.5]{FLT}). As another consequence, we can bound the Fitting height of a solvable group $G$ having a nilpotent subgroup $N$ such that $(G,N)$ is an equal order pair.

\begin{proposition}\label{fittingheight} Let $(G,N)$ be an equal order pair such that $G$ is solvable and $N$ is nilpotent. Then the Fitting height of $G$ is at most $4$.
\end{proposition}

\begin{proof}
We start by proving the following claim: if $H$ is a solvable group such that $\oh p H$ is cyclic or a generalized quaternion group for every prime divisor $p$ of $|H|$, then $H/\fit H$ has an abelian normal subgroup of index at most $2$ (hence, $H$ has Fitting height at most $3$). In fact, by \cite[III, 4.2]{huppert}, $H/\fit H$ acts faithfully on $\fit H/\frat H$ and therefore it can be identified with a subgroup of ${\rm{Aut}}(\fit H/\frat H)$; note that the automorphism group of $\oh p {\fit H/\frat H}$ is an abelian group for every odd prime, whereas it is either abelian or isomorphic to ${\rm S}_3$ for $p=2$. Thus ${\rm{Aut}}(\fit H/\frat H)$, which is the direct product of these automorphism groups for $p$ running over the prime divisors of $|H|$, is either abelian or the direct product of an abelian group with ${\rm S}_3$. Our claim immediately follows. 

Now, let $p$ be a prime divisor of $|N|$; then, by Proposition~\ref{p'semiregular} together with Lemma~1.5 of \cite{FLT}, the factor group of $G/N$ over $\oh p{G/N}$ (call it $H_p$) has cyclic or generalized quaternion Sylow subgroups for all the prime divisors of its order except $p$, hence $\fit{H_p}$ has the same property (but its order is clearly coprime to $p$). In other words, $H_p$ is a group as $H$ in the paragraph above, and we conclude that $H_p$ has Fitting height at most $3$. If $|N|$ is divisible by two distinct primes $p_1$ and $p_2$, then $G/N$ embeds in $H_{p_1}\times H_{p_2}$ and hence $G/N$ has Fitting height $3$; the desired conclusion that $G$ has Fitting height at most $4$ follows at once. On the other hand, if $N$ is a $p$-group then, setting $O/N=\oh p{G/N}$, we have that $O$ is a $p$-group; since $G/O=H_p$ has Fitting height $3$, the desired conclusion follows in this case as well and the proof is complete. 
\end{proof}

We note that the bound in Proposition \ref{fittingheight} is best possible since the Frobenius group whose Frobenius kernel is an elementary abelian group of order $49$ and a Frobenius complement is the non-split extension of ${\rm SL}_2 (3)$ by a cyclic group of order $2$ (this group is isoclinic to ${\rm GL}_2 (3)$, but is a Frobenius complement) has Fitting height $4$.  This raises the question of whether there is an equal order pair satisfying Proposition \ref{fittingheight} with Fitting height $4$ that is not a Frobenius group.

We also recall that, for a group $G$, the condition of being $p'$-semiregular over some field implies that every $p'$-subgroup of $G$ having order a product of two primes is cyclic (see \cite[Corollary~1.10]{FLT}); it turns out that the converse is also true if $G$ is solvable (\cite[Theorem~3.2]{FLT}). We refer the reader to Section~2 of \cite{FLT} for a thorough description of solvable $p'$-semiregular groups (thus, also of the solvable factor groups $G/N$ where $(G,N)$ is an equal order pair with $N$ nilpotent).

Finally, if $V$ is an $\F G$-module such that the centralizer in $V$ of every nontrivial $p'$-element of $G$ is trivial, then $(V \rtimes G, V)$ need not be an equal order pair; as an example we can consider $V$ to be the natural module of order $q^2$ for ${\rm{SL}}_2(q)$, where $q$ is a power of $2$ (see Remark~\ref{natural}).



\medskip

We noted in the Introduction that non-Frobenius Camina pairs $(G,N)$ with a nonsolvable factor group $G/N$ do exist.  The first example to be be found is described in \cite{nonsolv}.  Other examples are presented in Theorem 1.20 of \cite{GGLMNT}.  In this case, it is known that $N$ is necessarily a group of prime-power order.

Our aim in the rest of this section is to discuss equal order pairs $(G,N)$ such that {\it {$G/N$ is nonsolvable and $N$ is not a group of prime-power order}}. We will describe these equal order pairs in Theorem~C, and we will see in Remark~\ref{mixed} that there are examples of this kind. 

For an equal order pair $(G,N)$, it will follow from Theorem~D that the nonsolvability of $G/N$ implies the nilpotency of $N$; now, by Proposition~\ref{p'semiregular}, $G/N$ is a $p'$-semiregular group for every prime divisor $p$ of $|N|$ (in particular, every Sylow subgroup of $G/N$ is either cyclic or a generalized quaternion group) and the results of \cite{FLT} come into play. 

Recall that the set ${\mathcal{R}}$ is defined as the set of all primes $r$ satisfying both the following conditions.
\begin{enumeratei}
\item $r=2^a\cdot 3^b+1$, where $a\geq 2$ and $b\geq 0$.
\item $(r+1)/2$ is a prime.
\end{enumeratei}

Recall also that, by a classical theorem of Zassenhaus, a nonsolvable Frobenius complement has a subgroup of index at most $2$ which is isomorphic to ${\rm SL}_2 (5)\times M$, where $M$ is a metacyclic $\{ 2,3,5 \}'$-group.

\begin{ThmC}
Let $(G,N)$ be an equal order pair such that $G/N$ is nonsolvable and $N$ is not a group of prime-power order. Denote by $R/N$ and $K/N$ the solvable radical and the solvable residual of $G/N$, respectively. Then there exists a prime $p$ in $\mathcal{R}\cup\{7,17\}$ such that the following conclusions hold.
\begin{enumeratei}
\item $G/R$ is isomorphic either to ${\rm PSL}_2(p)$ or to ${\rm PGL}_2(p)$.
\item $K/N$ is isomorphic to ${\rm SL}_2(p)$.
\item If there exists a prime $r\in\pi(N)$ such that $K/N$ has a normal Sylow $r$-subgroup, then $p$ is necessarily $5$. Otherwise we have $\pi(N)=\{3,p\}$ when $p\neq 5$, whereas $\pi(N)\subseteq\{2,3,5\}$ when $p=5$.
\end{enumeratei}
\end{ThmC}

\begin{proof} 
Let $U/R$ be a minimal normal subgroup of $G/R$: clearly,  $U/R$ is nonsolvable. Recalling that the Sylow subgroups of $G/N$ are all cyclic or generalized quaternion groups by Proposition~\ref{nilpotent N}, observe first that the Sylow subgroups of $U/R$ for any odd prime are cyclic as well (because they are subgroups of quotients of cyclic groups); it easily follows that $U/R$ is a nonabelian simple group. 

Also, a Sylow $2$-subgroup $D/R$ of $U/R$ is a normal subgroup of a Sylow $2$-subgroup of $G/R$, which is in turn a quotient of a cyclic or generalized quaternion group; as $D/R$ cannot be cyclic, it is necessarily either a dihedral or a generalized quaternion group. In any case, $D/R$ is not a Suzuki $2$-group and hence $U/R$ is not a Suzuki simple group ${\rm Sz}(2^n)$ (note that, as a consequence, $|U/R|$ is divisible by $3$). Since the same argument applies to any other minimal normal subgroup $V/R$ of $G/R$, in $U/R \times V/R$ we could find a subgroup isomorphic to $C_3 \times C_3$, not our case. We conclude that $G/R$ has a unique minimal normal subgroup, hence it is an almost-simple group with socle $U/R$. 

Since $K$ is a perfect group, we have $KR=U$ and hence $K/(K\cap R)\cong U/R$ is a simple group. Set $\overline{K}=K/N$: adopting the bar convention, in view of Proposition~\ref{p'semiregular}, both $\overline{K}$ and $\overline{K}/\oh r{\overline{K}}$ are $r'$-semiregular groups for every prime divisor $r$ of $|N|$. 

Now, if there exists a prime $r$ in $\pi(N)$ such that $r$ does not divide $|\overline{K}/\oh r{\overline{K}}|$ (i.e., such that $\overline{K}$ has a normal Sylow $r$-subgroup) then, as observed in the paragraph preceding  Proposition~\ref{p'semiregular}, $\overline{K}/\oh r{\overline{K}}$ is a (perfect) Frobenius complement. It follows that $\overline{K}/\oh r{\overline{K}}$ is isomorphic to ${\rm SL}_2(5)$ (and $r \not\in \{ 2,3,5 \}$). Note that, if $\oh r{\overline K}$ is nontrivial, then there exists a normal subgroup $W$ of $G$ contained in $\oh r {\overline K}$ and such that $V = \oh r {\overline K}/W$ is a chief factor of $G$. Since $V$ is a cyclic group (recall that the Sylow subgroups of $G/N$ are all cyclic or generalized quaternion groups), its centralizer $C/W$ in $K/W$ is such that $K/C \leq {\rm {Aut}}(V)$ is abelian; but in fact  $C=K$ because $K$ is perfect. Our conclusion so far is that $V$ is contained in the Schur multiplier of $\overline{K}/\oh r {\overline{K}} \cong {\rm {SL}}_2 (5)$, which is trivial. This contradiction shows that $\oh r {\overline{K}}$ is trivial, hence $K/N \cong {\rm {SL}}_2 (5)$. The three conclusions (a), (b) and (c) all follow in this case.

In view of the previous discussion, we will henceforth assume that $K/N$ does not have a normal Sylow $r$-subgroup for any prime $r$ in $\pi(N)$. As a consequence, for every such prime, we can apply Theorem~4.1 of \cite{FLT} to $\overline{K}/\oh r{\overline{K}}$. Since $K/(K\cap R)$ is simple and not a Suzuki group, we immediately discard conclusions (ii) and (iii) of Theorem~4.1. 

On the other hand, assume that there exist two distinct primes $p,q\in\pi(N)$ for which conclusion (i) of Theorem~4.1 holds, i.e., $\overline{K}/\oh p{\overline{K}}\cong{\rm{SL}}_2(p^a)$ and $\overline{K}/\oh q {\overline{K}} \cong {\rm{SL}}_2 (q^b)$; this implies ${\rm {PSL}}_2 (p^a) \cong K/(K\cap R) \cong {\rm {PSL}}_2 (q^b)$, which occurs if and only if (up to relabelling) $p^a = 5$ and $q^b = 2^2$. In this case, as above, we obtain $K/N \cong {\rm {SL}}_2 (5)$; also, the assumption at the beginning of the previous paragraph yields $\pi (N) \subseteq \{ 2,3,5 \}$ and all the conclusions (a), (b), (c) are proved. Note that, in principle, $\pi(N)$ can be the full set $\{ 2,3,5 \}$ because the prime $3$ can be recovered as in conclusion (iv) of Theorem~4.1 (with $5$ in the role of $r$).

The only possibility which is left is that $\overline{K}/\oh 3{\overline{K}}$ is $3'$-semiregular as in conclusion (iv) of Theorem~4.1, so $\overline{K}/\oh 3{\overline{K}}$ is isomorphic to ${\rm {SL}}_2(p)$ for a prime $p$ in $\mathcal{R}\cup\{7,17\}$. There must be one (and only one) more prime in \(\pi(N)\), which is necessarily $p$, so we also have $\overline{K}/\oh p{\overline{K}}\cong{\rm {SL}}_2(p)$. It immediately follows that $\oh p{\overline{K}}=\oh 3{\overline{K}}=1$, so that $K/N\cong {\rm {SL}}_2(p)$, and the proof is complete.
\end{proof}

\begin{remark}\label{mixed} Note that equal order pairs $(G,N)$ as in Theorem~C, such that $K/N$ does not have a normal Sylow subgroup for any prime in $\pi(N)$, do exist. In fact, ${\rm SL}_2(5)$ has an irreducible module $W_1$ of order $3^4$ on which it acts $3'$-semiregularly; denoting by $W_2$ the natural module of order $5^2$, it can be checked (via \texttt{GAP}, for instance) that the semidirect product $G = (W_1\oplus W_2) \rtimes {\rm SL}_2(5)$ which is formed accordingly, together with its Fitting subgroup $N = W_1 \oplus W_2$, form an equal order pair. 
\end{remark}

We close the section showing an infinite family of equal order pairs $(G,N)$ where $G/N$ is nonsolvable. As already mentioned, it seems much harder to produce examples of Camina pairs with the same property.

\begin{remark}\label{natural}
Let $q\neq 1$ be a power of an odd prime $p$, and let $V$ be the natural module (of dimension $2$ over ${\rm{GF}}(q)$) for ${\rm {SL}}_2 (q)$. Defining $G$ to be the semidirect product formed accordingly, we claim that $(G,V)$ is an equal order pair.  Clearly, it will be enough to check that the coset $xV$ contains elements of equal order for every $x \in {\rm{SL}}_2(q)$.

In fact, since the centralizer in ${\rm{SL}}_2 (q)$ of any nontrivial element of $V$ is a Sylow $p$-subgroup of ${\rm{SL}}_2 (q)$, for every nontrivial $p'$-element $x \in {\rm{SL}}_2 (q)$ we have that $V \langle x \rangle$ is a Frobenius group with kernel $V$, hence the order of every element in the coset $xV$ is the same as the order of $x$. 

Assume next that the element $x \in {\rm{SL}}_2(q)$ is a $p$-element (recall that the Sylow $p$-subgroups of ${\rm{SL}}_2 (q)$ are elementary abelian).  Up to conjugacy in $G$, we can assume that $x$ is a lower-unitriangular matrix in ${\rm{SL}}_2 (q)$; in other words there exists an element $a \in {\rm{GF}} (q)$ such that, for every vector $v = (v_1,v_2) \in V$ (and using the additive notation for $V$), we have $v^x = (v_1 + a v_2, v_2)$. Now, we obtain  
$$(xv)^p = v + v^{x^{-1}} + v^{x^{-2}} + \cdots + v^{x^{-(p-1)}} = \left( pv_1 - \left( \frac{p(p-1)}{2} \right) av_2, \; pv_2 \right) = 0.$$

Finally, recall that any element of ${\rm{SL}}_2 (q)$ whose order is divisible by $p$ is of the form $zx$, for some element $z \in \zent{{\rm{SL}}_2 (q)}$ and $x$ is a (nontrivial) $p$-element.  Assuming we are in this setting, it remains to show that $zxv$ has the same order as $zx$ when $z$ is the unique involution of ${\rm{SL}}_2 (q)$. But $(zxv)^2 = x^2 [x,v] \in x^2 V$ has order $p$ by the paragraph above; noting that the order of $zxv V = zx V$ is $2p$, it immediately follows that the order of $zxv$ is $2p$, as desired.

On the other hand, the above construction does not yield an equal order pair when $q$ is a power of $2$. In fact, a Sylow $2$-subgroup of $G = V\rtimes{\rm{SL}}_2 (q)$ (where $|V| = q^2$) has at least one element of order $4$, whose coset modulo $V$ is represented by an element of order $2$ in ${\rm {SL}}_2 (q)$.

We note that when $q = 3$, we have that $G$ is ${\rm SL}_2 (3)$ acting on its natural module $N$, and this is the only case in this family where $G$ is solvable.  This gives us an example of a solvable equal order pair $(G,N)$ where $N = \fit G$ is a $p$-group for some prime $p$ and $G$ is not a Frobenius group.  This is the only example of an equal order pair with this property that we know of.  The remaining pairs in this family satisfy $N = \fit G =\oh p G$ and $G/N$ is nonsolvable.  We will remind the reader that Chillag, Mann, and Scoppola (\cite{CMS}) prove that any Camina pair where $N$ is a $p$-group will have that $N < \oh p G$. 
\end{remark}

\section{When $N$ is not nilpotent}\label{Sect 6}

We have already seen examples of equal order pairs $(G,N)$ where $N$ is not nilpotent (for instance, the one presented in the paragraph following Lemma~\ref{generalize j}). In this section, we will first show that the assumption of non nilpotency on $N$ forces $G/N$ to be a $p$-group for a prime number $p$ (in fact, it turns out that all elements of $G\setminus N$ are $p$-elements).  After that, we will discuss the situation when $N$ is {\it{nonsolvable}}, which, unlike Camina pairs, can occur.   

\smallskip
The following preliminary result will be relevant for our purposes.

\begin{lemma}
\label{NoTypeb}
Let $(G,N)$ be an equal order pair, and let $q$ be a prime such that $G/N$ has a proper nontrivial normal Sylow $q$-subgroup.  Then $N$ is nilpotent.
\end{lemma}

\begin{proof}
We argue by induction on the order of $G$. Let $Q$ be a Sylow $q$-subgroup of $G$: since our assumptions imply that  $QN$ is normal in $G$, by the Frattini argument we have $G=\norm GQN$. In particular $\norm G Q$ is not contained in $N$ and, unless $\norm NQ$ is trivial, Lemma~\ref{subgroups} yields that $(\norm GQ,\norm NQ)$ is an equal order pair. Note that, as we have $\norm G Q/\norm NQ\cong G/N$, this equal order pair satisfies our hypothesis as well.  

Now, since $G/N$ has order divisible by a prime $p\neq q$, we can find a $p$-element $x$ in $\norm GQ\setminus N$ such that the order of $xN$ is $p$. If $x$ has order $p$, then we are done by Proposition~\ref{pq}(b). On the other hand, assume $x$ does not have order $p$, so $x^p$ is a nontrivial $p$-element of $\norm NQ$ (which is therefore nontrivial). If $\norm GQ$ is a proper subgroup of $G$, then by induction $\norm NQ$ is nilpotent; hence $\norm NQ$ has a normal Sylow $p$-subgroup $P_0$ (clearly containing $x^p$), which is also normal in $\norm G Q$. As a consequence, we get $[x^p,Q]  \subseteq P_0\cap Q=1$. Now, recalling that $Q\not\subseteq N$ and taking an element $y\in Q\setminus N$, the nontrivial $p$-element $x^p\in N$ centralizes the $q$-element $y$ not in $N$, against Lemma~\ref{two}.

It remains to consider the case when $Q$ is normal in $G$. But in this situation we have $Q\subseteq \fit G$, hence $\fit G$ cannot be contained in $N$. Then Theorem~A yields that $N$ must be contained in $\fit G$, so $N$ is nilpotent and the proof is complete. 
\end{proof}
	
	


\smallskip
We are now ready to prove Theorem~D, that we state again for the convenience of the reader.  Following the literature, a group $G$ is a {\it {$2$-Frobenius group}} if there are normal subgroups $1 < K < N < G$ so that $N$ is a Frobenius group with Frobenius kernel $K$ and $G/K$ is a Frobenius group with Frobenius kernel $K/N$.  It is not difficult to see that if $G$ is a $2$-Frobenius group, then $G$ is solvable and its second Fitting subgroup ${\bf F}_2 (G)$ is a Frobenius group with kernel $\fit G$, and $G/\fit G$ is a Frobenius group with kernel ${\bf F}_2 (G)/\fit G$,

\begin{ThmD}
Let $(G,N)$ be an equal order pair such that $N$ is not nilpotent. Then the following conclusions hold.
\begin{enumeratei}
\item There exists a prime number $p$ such that every element in $G\setminus N$ is a $p$-element; in particular, $G/N$ is a $p$-group. 
\item If $P$ is a Sylow $p$-subgroup of $G$, then $(G,P,P\cap N)$ is a Frobenius-Wielandt triple.
\item If $N$ is nonsolvable, then $p=2$ and every nonsolvable chief factor $U/V$ of $G$ such that $U\subseteq N$ is isomorphic to ${\rm{PSL}}_2(3^{2^r})$ for some $r>0$.
\end{enumeratei}
\end{ThmD}

\begin{proof} 
If we prove that $G/N$ is a $p$-group for a prime number $p$, then the rest of conclusion (a) follows by Lemma~\ref{two}, whereas conclusion (b) is an application of Lemma~\ref{seven} (taking into account that, if $P$ is a Sylow $p$-subgroup of $G$, then $G=PN$ in this situation); therefore, we work to show that $G/N$ has prime-power order. Since we know that every element in $G\setminus N$ has prime-power order, we will be able to take advantage of some results in \cite{lewis}.

In particular, assume that $|G/N|$ is divisible by exactly two distinct primes $p$ and $q$: then Theorem~1.2 of \cite{lewis} yields that $G$ is a $\{p,q\}$-group, and $G/N$ is either a Frobenius group or a {{$2$-Frobenius group}}. 
We can consider the subgroup $H$ of $G$ such that $H/N = {\bf F}_2(G/N)$ (note that $H=G$ if $G/N$ is a Frobenius group); now, $(H,N)$ is an equal order pair by Lemma~\ref{subgroups}, and $H/N$ is a Frobenius $\{p,q\}$-group. But then, the Fitting subgroup of $H/N$ is a proper nontrivial normal Sylow subgroup of $H/N$, and Lemma~\ref{NoTypeb} yields the contradiction that $N$ is nilpotent.

Thus, it remains to treat the case when $|G/N|$ is divisible by at least three distinct primes. By \cite[Theorem~1.3]{lewis} this condition implies that all the elements of $G$ have prime-power order, and the structure of $G$ is described in \cite[Theorem~2.2]{lewis} (a theorem by R. Brandl). Since $G$ is clearly not a simple group, we are left with two possibilities: either $G$ is isomorphic to the Mathieu group ${\rm M}_{10}$, or $G$ has a normal elementary abelian $2$-subgroup $M$ such that $G/M$ is isomorphic to one of ${\rm  PSL}_2(4)$, ${\rm PSL}_2(8)$, ${\rm Sz}(8)$ or ${\rm Sz}(32)$. Now, if $G\cong{\rm M}_{10}$ then it has a unique proper, nontrivial, normal subgroup $N$ (which is $G'\cong{\rm A}_6$); but $|G/N|=2$, so all elements of $G\setminus N$ lie in a single coset modulo $N$ and among them there are both elements of order $4$ and elements of order $8$, against our hypothesis. On the other hand, assume that $G$ has a normal elementary abelian $2$-subgroup $M$ such that $G/M$ is a simple group as listed above. In this case, since $N$ is not nilpotent, we know by Theorem~A that $\fit G$ is (properly) contained in $N$, and in particular we have $M\subseteq N$; but now $N$ cannot be proper in $G$, and this final contradiction shows that $G/N$ has prime-power order. As argued before, conclusions (a) and (b) are proved.

Finally, conclusion (c) follows at once by Proposition~\ref{pq} (c) together with Theorem~A of \cite{SJ}.
\end{proof}



\medskip

We note that the only Camina pairs $(G,N)$ that are known to occur where $N$ is not nilpotent are groups $G$ where $G$ is a Frobenius group whose Frobenius complements are quaternion groups of order $8$ and $N = G'$ is the unique normal subgroup of index $4$ and it contains the Frobenius kernel as a subgroup of index $2$.  It is an open question as to whether there are any other Camina pairs where $N$ is not nilpotent.  In particular, it is not known if such Camina pairs do exist, and also whether there is a bound on the Fitting height of $N$.  We mention that the example displayed after Lemma \ref{generalize j} and examples of this kind provide equal order pairs $(G,N)$ where $G$ is solvable and $N$ is not nilpotent.  We note that the examples we provide only have $N$ with Fitting height $2$, but we also do not have a bound on the Fitting height of $N$.  Hence, it is an open question as to whether or not we can bound the Fitting height of $N$ when $(G,N)$ is an equal order pair when $N$ is solvable, and if we can if that bound would be larger than~$2$.

\medskip
To conclude this paper, we focus on equal order pairs $(G,N)$ where $N$ is nonsolvable. We start by observing that such pairs do exist: an example is given by a group $G$ which is a split extension of ${\rm A}_6$ by a cyclic group of order $8$ (namely, the action of $C_8$ on ${\rm A}_6$ has a kernel of order $4$ which is $\zent G$, and $G/\zent G$ is isomorphic to the Mathieu group ${\rm M}_{10}$). Denoting by $N$ the unique subgroup of index $2$ of $G$, which is isomorphic to ${\rm A}_6\times C_4$, it can be checked that every element of $G$ lying outside $N$ has order $8$. So, $(G,N)$ is an equal order pair with a nonsolvable $N$ and, as we will see in Proposition~\ref{minimalnonsolvable}, this is the smallest possible example of this kind. 





In what follows, when we say that the group $G$ is minimal with respect to the property of having a nonsolvable subgroup $N$ such that $(G,N)$ is an equal order pair, we mean that $G$ does not have any subgroups or factor groups satisfying the same property. Note that, in view of Theorem~D(c) and Lemma~\ref{subgroups}, in such a situation we have $|G:N| = 2$.


\begin{lemma}\label{4.3} 
Let $G$ be a group which is minimal with respect to the property of having a nonsolvable subgroup $N$ such that $(G,N)$ is an equal order pair. Let $M$ be a normal subgroup of $G$ contained in $N$ and $x\in G \setminus N$. Then either $M$ is solvable, or $G = M \langle x \rangle$ and $N = M \langle x^2 \rangle$.
\end{lemma}

\begin{proof} 
We already observed that $|G/N|=2$, and Lemma~\ref{two} yields that $x$ is a $2$-element. Taking into account that $M\langle x^2\rangle= N\cap M\langle x\rangle$, Lemma~\ref{subgroups} yields that $(M\langle x\rangle,M\langle x^2\rangle)$ is an equal order pair (unless $M\langle x^2\rangle$ is trivial, but in this case there is nothing to prove). If $M \langle x\rangle$ is a proper subgroup of $G$, then our minimality assumption implies that $M\langle x^2\rangle$ is solvable and so is $M$; otherwise we get $G=M\langle x\rangle$ and $M\langle x^2\rangle=N\cap M\langle x\rangle=N\cap G=N$, as claimed. 
\end{proof}


We are ready to prove the main result of this section.

\begin{proposition}\label{minimalnonsolvable}  
Let $G$ be a group which is minimal with respect to the property of having a nonsolvable subgroup $N$ such that $(G,N)$ is an equal order pair, and let $R$ be the solvable radical of $G$. Then $G/R$ is isomorphic to the Mathieu group ${\rm M}_{10}$, and $R$ is not a minimal normal subgroup of $G$.
\end{proposition} 

\begin{proof} 
Since $|G:N|=2$, by Lemma~\ref{two} every element lying in $G\setminus N$ is a $2$-element; therefore, by Theorem~A of \cite{SJ}, every nonsolvable chief factor $H/K$ of $G$ with $H\subseteq N$  is a simple group, isomorphic to a group of type ${\rm{PSL}}_2(3^r)$ where $r$ is a $2$-power. Note also that every proper subgroup $X$ of $G$ such that $X\not\subseteq N$ is solvable: in fact, by Lemma~\ref{subgroups}, $(X,X\cap N)$ is an equal order pair (unless $X\cap N=1$, but in this case $|X|=2$), and therefore our minimality assumption implies that $X\cap N$ is solvable. Now, as $|X/X\cap N|=2$, $X$ is solvable as well.

Let $R$ be the solvable radical of $G$. Setting $L=N\cap R$, clearly $L$ is a normal subgroup of $G$ which is properly contained in $N$; therefore, we can consider a normal subgroup $M$ of $G$ such that $L\subseteq M\subseteq N$ and $M/L$ is a (nonsolvable) chief factor of $G$. Now, take $x$ in $G\setminus N$; the fact that $M$ is nonsolvable together with Lemma~\ref{4.3} yields $G=M\langle x\rangle$ and $N=M\langle x^2 \rangle$. Setting $\overline{G}=G/L$, we will adopt the bar convention.

Assuming for the moment that $R$ is not contained in $N$, we can choose $x\in G\setminus N$ to be an element of $R$; but since $x^2$ lies in $N\cap R=L\subseteq M$, we get $N=M\langle x^2\rangle=M$. 
Observe that, to get a contradiction, it is enough to produce an element lying in $\overline{G}\setminus\overline{N}$ whose order is divisible by an odd prime. 
Set $\overline{C}=\cent {\overline{G}}{\overline{N}}$: if $\overline{C}\neq 1$, then we get $\overline{G}=\overline{N}\times \overline{C}$, hence $\overline{C}$ is generated by an element $\overline{c}$ of order $2$. So, taking $\overline{a}\in\overline{N}$ whose order is an odd prime $p$, we get ${\rm{o}}(\overline{ac})=2p$ (and of course $\overline{ac}\in\overline{G}\setminus\overline{N}$). We conclude that $\overline{C}$ must be trivial, but $\overline{C}$ contains $\overline{R}$ and we obtain the contradiction that $R=L$.

Our conclusion so far is that $R$ is contained in $N$ (hence, in $M$). As above, we consider an element $x\in G\setminus N$  and we have $G=M\langle x\rangle$; in particular, $G/M=M\langle x\rangle/M$ is a cyclic $2$-group. Still defining $\overline{G}=G/R$, we claim that  
 $\overline{M}$ is the unique minimal normal subgroup of $\overline{G}$. In fact, if $\overline{W}$ is another minimal normal subgroup of $\overline{G}$, then $\overline{W}\cong WM/M$ is cyclic, hence solvable, and this yields the contradiction $W\subseteq R$. As already observed, by \cite{SJ} we have that $\overline{M}$ is a simple group isomorphic to ${\rm{PSL}}_2(3^r)$ where $r$ is a $2$-power, and $\overline{G}$ is an almost simple group. 



The discussion in the first paragraph of this proof implies that every proper subgroup of $\overline{G}$ not contained in $\overline{N}$ is solvable. We claim that, as a consequence, $\overline{M}$ is a simple group belonging to the following list $\mathcal{L}$.

\begin{enumeratei}
\item ${\rm{PSL}}_2(q)$ for $q\in\{2^r,3^r,p^{2^a}\}$, where $r$ is a prime, $p$ is an odd prime and $a\geq 0$,
\item ${\rm{PSL}}_3(3)$,
\item ${\rm{Sz}}(2^s)$, where $s$ is an odd prime.
\end{enumeratei}

In fact, let $\overline{G}$ be an almost simple group whose socle $\overline{M}$ is not in the above list. Set $A={\rm{Aut}}(\overline{M})$. By Lemma~3.2 of \cite{FLP}, $\overline{M}$ has a nonsolvable proper subgroup $\overline{H}$ such that $\overline{M}\norm A{\overline{H}}=A$; by Lemma~3.1 of the same paper, we also have $\overline{G}=\overline{M}\norm{\overline{G}}{\overline{H}}$. 
Now, $\norm{\overline{G}}{\overline{H}}$ is a nonsolvable proper subgroup of $\overline{G}$ not contained in $\overline{N}$, a contradiction. We conclude that, in our situation, $\overline{M}$ belongs to the list $\mathcal{L}$.

But we also know that $\overline{M}$ is isomorphic to ${\rm{PSL}}_2(3^r)$ where $r$ is a $2$-power (see also \cite[Proposition~16]{LO}); putting this together with our previous discussion, it follows that the only possibility for $\overline{M}$ is to be isomorphic to ${\rm{PSL}}_2(3^2)\cong {\rm A}_6$. Our conclusion so far is that $\overline{G}$ is an almost simple group whose socle is isomorphic to ${\rm A}_6$, and whose elements outside the socle all have a $2$-power order: it can be checked (via \texttt{GAP} or the ATLAS \cite{atlas}, for instance) that $\overline{G}$ is then isomorphic to the Mathieu group ${\rm M}_{10}$.

\smallskip
Finally, for a proof by contradiction, assume that $R$ is a minimal normal subgroup of $G$. If $R$ has odd order, then we claim that $(G/R, N/R)$ is an equal order pair: in fact, for every element $xR$ outside $N/R$ we have ${\rm o}(xR)={\rm o}(x)$ (recall that $x$ is a $2$-element), and since all elements in $G\setminus N$ have the same order, our claim follows. Now, our minimality assumption on $G$ yields $R=1$, but then we would have $G\cong{\rm M}_{10}$ and we have already seen that this is not possible. We deduce that $R$ is an elementary abelian $2$-group, so $R$ can be viewed as an irreducible module for $G/R\cong {\rm M}_{10}$ over ${\rm GF}(2)$. As $G/N\cong C_2$ cannot act faithfully and irreducibly on the $2$-group $R$, either $R$ is central in $G$ or the action of $G/R$ on $R$ is faithful. 

Assume first $R\subseteq\zent G$. Then $R$ cannot be contained in $G'$, as otherwise it would be isomorphic to a subgroup of the Schur multiplier of ${\rm M}_{10}$, which has order $3$. Therefore $R$ is not contained in $N'$, and we get $R\cap N'=1$ because $R$ is minimal normal in $G$. Moreover, $N/R\cong {\rm A}_6$ being a perfect group, we get $N/R=(N/R)'=N'R/R$, and it follows that $N=N'\times R$. It is known that ${\rm M}_{10}$ has elements, of orders $4$ and $8$ respectively, that do not lie in ${{\rm M}_{10}}'\cong {\rm A}_6$: thus we can consider $g$, $h$ in $G\setminus N$ with ${\rm o}(gR)=4$ and ${\rm o}(hR)=8$ and, to get a contradiction, it will be enough to show that ${\rm o}(g)=4$ as well (in fact, we would have $gN=hN$ with ${\rm o}(g)=4\neq 8\leq {\rm o}(h)$). Observe that $g^2$ lies in $N$, hence $g^2=uv$ for suitable elements $u\in N'$ and $v\in R$; clearly we get ${\rm o}(u)={\rm o}(uR)={\rm o}(g^2R)=2$ and, since $v^2=1$, we have ${\rm o}(g^2)=2$. The aforementioned contradiction ${\rm o}(g)=4$ follows.

Therefore $R$ is necessarily a faithful irreducible module for ${\rm M}_{10}$. It can be checked via {\texttt{GAP}} that (up to isomorphism) there exist two such modules, having orders $2^8$ and $2^{16}$ respectively; moreover, if $R$ is any of these modules, an extension of $R$ by ${\rm M}_{10}$ is necessarily a split extension. In other words, $G$ is isomorphic to a semidirect product $R\rtimes {\rm M}_{10}$. Now, taking two elements in ${\rm M}_{10}\setminus {{\rm M}_{10}}'$ of orders $4$ and $8$ as above, we easily get the final contradiction that completes the proof.
\end{proof}

\end{document}